\documentclass{jaspre}

\usepackage{color}  

\definecolor{red-}{rgb}{1.0,0.0,0.0}
\definecolor{blue-}{rgb}{0.0,0.2,1.0}
\definecolor{bluenight-}{rgb}{0.0,0.1,0.55}
\definecolor{green-}{rgb}{0.0, 0.6, 0.0}
\definecolor{gold}{rgb}{0.8,0.7,0.0}
\definecolor{black}{rgb}{0.0,0.0,0.0}
\definecolor{DarkGreen}{rgb}{0.0,0.5,0.0}
\definecolor{LightGreen}{rgb}{0.8,1.0, 0.8}
\definecolor{yellow}{rgb}{0.9,0.9,0.0}
\definecolor{white}{rgb}{1.0,1.0,1.0}

\begin{document}

\title{Cubature rules and expected value of some complex functions}

\author[C. Fassino, E. Riccomagno, M-P Rogantin]{C. Fassino\affil{1}, E. Riccomagno\affil{1}\comma\corrauth, M-P Rogantin\affil{1}}

\address{\affilnum{1}\ Department of Mathematics, University of Genova, Italy}

\emails{{\tt fassino@dima.unige.it}\ (C. Fassino), 
{\tt riccomagno@dima.unige.it}\ (E. Riccomagno),
{\tt rogantin@dima.unige.it}\ (M-P. Rogantin)
}

\begin{abstract}
The expected value of some complex valued random vectors is computed by means of the indicator function of a designed experiment as known in algebraic statistics. 
The general theory is set-up and results are obtained for finite discrete random vectors and the Gaussian random vector. 
The precision space of some cubature rules/designed experiments are determined. 
\end{abstract}


\maketitle

\section {Introduction}

Evaluations of integrals is a recurrent task in statistics and probability for example when computing marginal distributions, in  the analysis of contingency tables, when estimating the moments of some known distribution or when evaluating the marginal likelihood integrals in Bayesian inference, spectral analysis of time series~\cite{Brillinger2001} and probability
Mora than in statistics, complex valued random vectors and their integration  find application in many other fields such as electromagnetism and quantum mechanics, 
and largely in  digital communication~\cite{Lapidoth2009} and signal processing (e.g.~\cite{ZhuBlumLin2016} and for a setting similar to ours~\cite{SaylorSmolarski2001}). 
Interestingly the usefulness of complex random vectors has also been argued in actuarial science~\cite{Halliwell2016} besides time series analysis. 
An introduction to the statistical analysis based on complex Gaussian distributions is given in~\cite{Goodman1963} and a recent paper on second order estimation with complex-valued data focused on digital signal processing can be found in~\cite{LangHuemer2017}. 
 
In this paper we address the problem of computing \emph{exactly} the expected value with respect to a generic probability measure $\lambda$, of a complex valued function $g: \mathbb C^k \rightarrow \mathbb C $ of the $k$-variate complex random vector $Z$.  
We do so by using the indicator functions from the algebraic statistics theory of design of experiments.  
The measure $\lambda$ could be discrete or continuous. General results are presented for the discrete case and specific results are given for the multivariate complex Gaussian distributions. 

The expected value of $g$ is approximated  by using interpolatory cubature rules of the form 
\begin{equation} \label{eq1} \int_{\mathbb C^k} g \, d\lambda = \sum_{d \in \mathcal D}  w_d g(d) +R(g) \end{equation}
where $ \mathcal D$ is a finite set giving the cubature nodes.
For us the  coordinates of the cubature nodes   are in suitable subsets  of the  $m$-th roots of the unity.  
The  weights $\{w_d\}_{d \in \mathcal D}$ are obtained  from a vectorial basis of the quotient space $\mathbb C[z_1,\dots,z_k] / I(\mathcal D)$, where  $I(\mathcal D)$ is  the polynomial ideal of $\mathcal D$~\cite{PRW2001}.
Finally $R(g)$ is the error committed when approximating the integral with the finite sum in~(\ref{eq1}). 
Given a set of nodes and weights, it is of interest to determine classes of functions $g$ for which the error is zero. This set is called the precision space of the cubature rule. 
Quadrature rules (i.e. bi-dimensional cubature rules) with complex valued nodes have been studied e.g. in~\cite{Panja2015Mandal}. 
Here we work in a multi-dimensional setting. 

Our work unveils a connection between cubature rules and design of experiments which, to our knowledge, has been unnoticed so far in the literature. 
We find this connection somewhat natural because both in cubature rule theory and design of experiment theory a key point is to determine a suitable finite set of points $\mathcal D$ and their weights $\{w_d\}_{d \in \mathcal D}$ for  achieving some specific task, although this can be different between the two theories and also within them.  
Another common  task to the two theories is, given $\mathcal D$ and $\{w_d\}_{d \in \mathcal D}$, find their range of applicability, e.g. power of estimation, precision space. 

This paper deals with this second task and it does so by the synergic use of tools and techniques from commutative algebra, numerical analysis and algebraic statistics. 
 In particular, some results in~\cite{FaPR2014,FaR2015} are generalised to the complex case.
The link between the above cubature problem and the algebraic statistics theory of fractional factorial design of experiments is made 
through  the representation of a fractional factorial experiment as a polynomial indicator function~\cite{PR2008}.
This is similar to~\cite{FRR2012} which instead unhearthed the connection between Markov bases for contingency tables and design of experiments.

In Section~\ref{Sc:Eq_weights} we focus our attention on the special case with equal weights and we obtain some specific results for the Gaussian density in Section~\ref{Sc:Gauss}.
While in Section~\ref{Sc:Interp_Rule} we provide necessary and sufficient conditions for obtaining such  cubature rules and we analyse their precision space, that is the vector space of polynomials $p$  whose expected value is equal to $\sum_{d \in \mathcal D}  w_dfp(d)$, namely with zero  error $R(p) $. The weights are found in Section~\ref{Sc:Comp_weights}.

\section{Interpolatory rules}\label{Sc:Interp_Rule}
Let $\lambda$ be a measure on $\mathbb C^k$ with finite moments (at least up to a certain degree) and $g$ be a complex integrable function, $g: \mathbb C^k \to \mathbb C$.
Let $\mathcal D\subset \mathbb C^k$ be a set with $n$ elements and let $w \in \mathbb C^k$ be the vector $[w_d]_{d\in \mathcal D}$.

A \emph{cubature rule} $(\mathcal D,w)$ is a formula of the type
  \begin{equation*}
  \int_{\mathbb C^k} g \ d\lambda=  \sum_{d \in \mathcal D} w_d \ g(d) +R_{\mathcal D,w}(g)
\end{equation*}
where the sum provides an approximation to the integral and $R_{\mathcal D,w}$ is the respective error.
The $w_d$'s are called the \emph{weights} and the elements of $\mathcal D$  the \emph{nodes} of the cubature rule.

Let $\mathbb C[z_1,\dots,z_k]$ be the ring of polynomials with complex coefficients in the indeterminates  $z_1,\dots,z_k$ and let $\mathcal P$ be  a set of polynomials contained in $\mathbb C[z_1,\dots,z_k]$. A cubature rule $(\mathcal D,w)$ is  \emph{exact} for $\mathcal P$ 
 if for all elements $p$ of $\mathcal P$
\begin{equation*}
\int_{\mathbb C^k} p \ d\lambda=\sum_{d \in \mathcal D} w_d \ p(d)
\end{equation*}
or, equivalently, if  $R_{\mathcal D,w}(p)=0$.

\vspace{3 mm}
A cubature rules $(\mathcal D,w)$ is called  \emph{interpolatory} if it is exact for a set $\mathcal P$ of interpolatory polynomials over $\mathcal D$. This definition generalises the definition of univariate interpolatory quadrature rules, where $\mathcal P$ is the set of univariate polynomials with degree strictly lower than the cardinality of $\mathcal D$, that is the set of the interpolatory polynomials over $\mathcal D$~\cite{Gautschi2004}. 

In this paper, given a set of nodes $\mathcal D$, we only consider sets of polynomials $\mathcal P$ such that,  
for any function $g: \mathbb C^k \to \mathbb C$,  there exists a \emph{unique} interpolatory polynomial $p_{g,\mathcal{D}}\in \mathcal P$  with $g(d)=p_{g,\mathcal{D}}(d)$ for all $d\in \mathcal D$. The pair $(\mathcal D,\mathcal P)$ is  called \emph{correct}.  For instance, the pair $(\mathcal D,\mathcal P)$ is  correct if $\mathcal D=\{d_1,\dots,d_n\}\subset \mathbb C$ and  $\mathcal P=\operatorname{Span}_{\mathbb C}\left( x^\alpha \ | \ \alpha=0,\dots,n-1\right)$.

\vspace{3mm}
Let $\alpha \in \mathbb Z_{\ge 0}^{k}$, let $z^\alpha = z_1^{\alpha_1} \dots z_k^{\alpha_k}$ be a monomial in the indeterminates  $z_1,\dots,z_k$ and let  $\mathbb T = \left\{z^\alpha \ | \ \alpha \in \mathbb Z_{\ge 0}^k   \right\}$
be the set of all monomials.
Let  $S \subset \mathbb T$ be a set of monomials such that each $p \in \mathcal P$ can be expressed as $p=\sum_{s\in S} c_s s$, $c_s \in \mathbb C$, that is $\mathcal P$ is a vectorial space over $\mathbb C$ with basis $S$. We denote $\mathcal P= \operatorname{Span}_{\mathbb C}(S)$.

An interpolatory polynomial $p_g \in \operatorname{Span}_{\mathbb C}(S)$ of a function $g$ over $\mathcal D$ is such that $p_g(d)= g(d)$, that is $\sum_{s\in S} c_s s(d)=g(d) $, for each $d \in \mathcal D$. Denoting by $X_{\mathcal D, S} =[s(d)]_{d\in \mathcal D, s\in S}$ the evaluation matrix of the elements of $S$ at $\mathcal D$ and by $[g(d)]_{d\in \mathcal D}$ the evaluation vector of $g$ at  $\mathcal D$, the coefficient vector $c=[c_s]_{s\in S}$ of the polynomial $p_g$ satisfies the linear system $X_{\mathcal D,S} c =[g(d)]_{d\in \mathcal D}$. If the pair $(\mathcal D, \mathcal P)$ is correct, for each function $g$, that is for each vector $[g(d)]_{d\in \mathcal D}$,  there exists a unique coefficient vector $c$ solution of this linear system, and so the pair is correct if and only if $X_{\mathcal D, S}$ is a square non singular matrix, that is if and only if $\# S = \# \mathcal  D$ and the evaluation vectors $[s(d)]_{d \in \mathcal D}$ are linear independent vectors. 

\vspace{3 mm}
In the following we consider cubature rules $(\mathcal D, w) $ which are interpolatory with respect to a polynomial set $\mathcal P$ such that the pair $(\mathcal D,\mathcal P)$ is  correct. 
A correct pair $(\mathcal D,\mathcal P)$ can  be obtained considering 
$\mathbb C$-vector spaces $\mathcal P$ which are isomorphic to the quotient space  $\mathbb C[z_1,\dots , z_k]/I(\mathcal D)$, that is  considering monomial sets $S$ isomorphic to a basis of the quotient space  $\mathbb C[z_1,\dots , z_k]/I(\mathcal D)$. 
{There exist algebraic algorithms to compute  monomial bases of such a vector space, for instance the Buchberger-M\"oller algorithm~\cite{BuchbergerMoller1982}}.

By definition,  an interpolatory cubature rule $ (\mathcal D,w)$ is exact for each polynomial in $\mathcal P$ but, in general, there exist polynomials $p \notin \mathcal P$ such that $R_{\mathcal D,w}(p)=0$. In order to study the set of these polynomials, we introduce the notions of precision basis and precision space.

 A finite monomial set $\mathcal B_{\mathcal D,w} \subset \mathbb T$ is a \emph{precision basis} for $(\mathcal D,w)$ if
\begin{equation*}
\int_{\mathbb C^k} z^\alpha \ d\lambda =\sum_{d \in \mathcal D} w_d z^\alpha(d) \ \text{ for all } \; z^\alpha \in \mathcal B_{\mathcal D,w}  \ \textrm{and} \
\int_{\mathbb C^k} z^\alpha \ d\lambda \neq  \sum_{d \in \mathcal D} w_d z^\alpha(d) \text{ for all }  z^\alpha \notin \mathcal B_{\mathcal D,w} \ .
\end{equation*}
The precision basis is the largest set of monomials for which $(\mathcal D,w)$ is exact.  
The \emph{precision space} of $(\mathcal D,w)$ is the $\mathbb C$-vector space $\operatorname{Span}(\mathcal B_{\mathcal D,w})$   generated by $\mathcal B_{\mathcal D,w}$. 
In the univariate case, if $\mathcal D$ is a subset of $ \mathbb R$ of cardinality  $n$ and if $\mathcal P$ is the interpolation space is generated  $\{1,x,x^2,\dots,x^{n-1}\}$,  the precision space of  Gaussian quadrature rule is generated by $\mathcal B_{\mathcal D,w}=\{1,x,x^2,\dots,x^{2n-1}\}$~\cite{Gautschi2004}.

In the univariate case, the precision degree of a quadrature rule is the maximal degree of the elements of $\mathcal P$ on which the quadrature rule is exact. Generalising this notion, we define the \emph{precision degree} of $(\mathcal D,w)$ as the $\max_{z^\alpha \in \mathcal B_{\mathcal D,w}} \left\{\sum_{i=1}^k \alpha_i \right\}$.

\section{Weights for points in $\mathbb C$-vector space with basis $S$}\label{Sc:Comp_weights}  

Let $\mathcal D$ be a set of $n$ nodes in $\mathbb C^k$, let $S \subset \mathbb T$ be a set of monomials in $\mathbb C[z_1,\dots,z_k]$  and let $\mathcal P=\operatorname{Span}_{\mathbb C}(S)$ be the $\mathbb C$-vector space of polynomials in $\mathbb C[z_1,\dots,z_k]$ generated by $S$ such that $(\mathcal P, \mathcal D)$ is correct.

The following proposition gives the vector of weights $w_S$ that makes $(\mathcal D,w)$ exact on $\mathcal P$, 
{that is the weights $w_S$ of  the interpolatory cubature rule on $\mathcal P$}.

\begin{proposition} \label{pr:int-cub-r}
  Let $\mathcal P$ be a $\mathbb C$-vector space with basis $S$ with $\#S = \# \mathcal D$, and let $(\mathcal D,\mathcal P)$ be correct. Let  $X_{\mathcal D,S}=[s(d)]_{d \in \mathcal D,s \in S}$  be the evaluation matrix of the elements of  $S$ over $\mathcal D$.
  The cubature rule $(\mathcal D,w_S)$ is exact on $\mathcal P$ if and only if
$$
w_S=\left(X_{\mathcal D,S}^t\right)^{-1} \left[\int_{\mathbb C^k} s \ d\lambda \right]_{s\in S} \ .
$$
Furthermore, the weights $w_S$ are unique.
\end{proposition}

\begin{proof}
 Each  $p \in \mathcal P$ can be written uniquely as
  $$
  p=\sum_{s\in S} c_s s \ , \textrm{ for } c_s \in \mathbb C \ ,
  $$
and so
  $$
  [p(d)]_{d \in \mathcal D}=X_{\mathcal D,S} [c_s]_{s \in S} \ ,
  $$
  that is 
   \begin{equation*}
[c_s]_{s \in S}= X_{\mathcal D,S}^{-1} [p(d)]_{d \in \mathcal D} \ .
  \end{equation*}
It follows that, for each $p \in \mathcal P$,  
  \begin{align*}
   \int_{\mathbb C^k} p\ d\lambda&= \int_{\mathbb C^k} \sum_{s\in S} c_s\ s \ d\lambda= \sum_{s\in S} c_s \int_{\mathbb C^k} s \ d\lambda=\left[\int_{\mathbb C^k} s \ d\lambda \right]_{s\in S}^t \left[c_s\right]_{s \in S} \\ & =
\left[\int_{\mathbb C^k} s \ d\lambda \right]_{s\in S}^t \ X_{\mathcal D,S}^{-1} \ \left[p(d)\right]_{d \in \mathcal D}
  \end{align*}

The cubature rule $(\mathcal D,w_S)$ is exact on $\mathcal P$ if and only if
$\int_{\mathbb C^k} p\ d\lambda=w_S^t \left[p(d)\right]_{d \in \mathcal D}$ for each $p \in \mathcal P$, 
that is if and only if 
$$ \left[\int_{\mathbb C^k} s \ d\lambda \right]_{s\in S}^t \ X_{\mathcal D,S}^{-1} \ \left[p(d)\right]_{d \in \mathcal D} = 
 w_S^t \left[p(d)\right]_{d \in \mathcal D}$$
or, equivalently,  
  \begin{equation*}
  w_S= \left(X_{\mathcal D,S}^t \right)^{-1}  \ \left[\int_{\mathbb C^k} s \ d\lambda \right]_{s\in S} +\rho
  \end{equation*}
  where $\rho$ is orthogonal to each evaluation vector $\left[p(d)\right]_{d \in \mathcal D} $. In particular, $\rho$ is orthogonal to $\left[s(d)\right]_{d \in \mathcal D} $, $ s\in S$, that is to the columns of $X_{\mathcal D,S}$. Since $X_{\mathcal D,S}$ is a square non singular matrix, $\rho $ is the null vector and so  the 
  {vector $w_S$ } of the weights of the cubature rule $(\mathcal D, \operatorname{Span}(S))$ is unique.
\end{proof}

\begin{remark}
The weights $w_S$ do not change, if a different basis $T$ for the $\mathbb C$-vector space {$\mathcal P=\operatorname{Span}(S)$} is chosen. Each monomial $t \in T$ can be expressed as $t=\sum_{s \in S} m_{t,s} s$, and so, denoting by $M=[m_{t,s}]_{t \in T, s \in S}$, we have, by linearity,  
  $$
  \left[\int_{\mathbb C^k} t \ d\lambda \right]_{t \in T}=M\left[\int_{\mathbb C^k} s \ d\lambda \right]_{s\in S}\ .
  $$
Moreover, the evaluation matrix $X_{\mathcal D,T}$ of the elements of $T$ over $\mathcal D$ can be written as $X_{\mathcal D,T}=X_{\mathcal D,S}M^t$.  
Let $w_T$ be the weights computed using $\mathcal D$ and $T$. From Proposition~\ref{pr:int-cub-r} we have
  \begin{align*}
  w_T & =  \left(X_{\mathcal D,T}^t\right)^{-1}  \ \left[\int_{\mathbb C^k} t \ d\lambda \right]_{t\in T} 
 = \left( M X_{\mathcal D,S}^t \right)^{-1} M \left[\int_{\mathbb C^k} s \ d\lambda \right]_{s\in S} \\ 
 & = \left(X_{\mathcal D,S}^t \right)^{-1} \left[\int_{\mathbb C^k} s \ d\lambda \right]_{s\in S} = w_S\ .
  \end{align*}
  \end{remark}

\section{$\mathbb C$-Fractional factorial designs}
 In this section we consider interpolatory cubature rules with set of nodes $\mathcal D$, whose elements are $k$-uple of $m$-th roots of the unit, and with interpolatory space $\mathcal P$, the $\mathbb C$-vector space generated by a monomial set $S \subset \mathbb T$, such that the pair $(\mathcal D ,\mathcal P)$ is correct. 

We briefly recall some topics about of the roots of the unity. Let $m \in \mathbb Z_{\ge0}$ and   $\Omega_m=\{\omega_0,\dots, \omega_{m-1}\}$ be the set of the $m$-th roots of the unity, $\omega_j=\exp (-\mathbf{i} (m/2\pi)j)$, where $\mathbf{i}=\sqrt{-1}$ is the imaginary unity.

Denoting, for $j \in \mathbb Z$, by  $[j]_m$  the residue of $j \mod m$ and  by  $\overline j$ the class $[m-j]_m$ we have that, given   $c \in \mathbb Z$ and $\omega_j, \omega_i \in \Omega$,   it holds $\omega_j^c = \omega_{[c j]_m}$, $\omega_i\omega_j=\omega_{[i+j]_m}$, and the complex  conjugate of $\omega_j$ is  $\overline{\omega_j}=\omega_{\overline j}$. Furthermore, we denote by $ \mathbb Z_m$  the  set of all congruence classes of the integers for a modulus $m$, and by $\mathbb Z_m^k$  its cartesian product. 

We consider a  set of $n$ nodes $\mathcal D$ contained in  $\Omega^k_m \subset \mathbb C^k$. Let $f$ be the indicator function of $\mathcal D$ over $\Omega_m^k$, defined as  $f(d)=1$, for $d \in \mathcal D$, and  $f(d)=0$, for $d \in \Omega_m^k\setminus \mathcal D$.

Let  $S_m=\{z^\alpha: \  \alpha \in \mathbb Z_m^k\}$ be the monomial basis of {the} $\mathbb C$-vector space which is isomorphic to the quotient space  $\mathbb C[z_1,\dots , z_k]/I(\Omega_m^k)$.
 As presented in~\cite{PR2008}, by interpolating the values $[f(d)]_{d \in \Omega_m^k}$ with polynomials in $\operatorname{Span}(S_m)$, we obtain a representation of the indicator function $f$ as follows
 \begin{equation}\label{ind}
f = \sum_{\alpha \in \mathbb Z_m^k} b_\alpha z^\alpha \qquad \textrm{where} \quad b_\alpha = \frac 1 {m^k} \sum_{d \in {\mathcal D}} z^{\overline \alpha}(d) 
\end{equation}
and, since  $\alpha= [\alpha_1, \dots, \alpha_k]\in \mathbb Z_m^k$, $\overline \alpha= [m-\alpha_1, \dots, m-\alpha_k] $.


Let $(\mathcal D, w_S)$ be the cubature rule with nodes $\mathcal D\subset \Omega_m^k$, $S\subset S_m$  and weights $w_S = (X_{\mathcal D,S}^t)^{-1} \left[ \int_{\mathbb C^k} s \ d\lambda \right]_{s\in S}$.


In the next we consider  non-negative and normalized weights, i.e. such that $w_S^t1_n =1$, where $1_n$ is a $n$-vector with elements equal to $1$, and the weights are all equal, that is $w_S= \frac1 n 1_n$.


\subsection{Equal weights} \label{Sc:Eq_weights}
If the weights  are non-negative, normalized and all equal, then $w = \frac 1 n {1_n}$. 
The following theorem characterises a cubature rule with equal weights using the indictor function of $\mathcal D$.

\begin{theorem}\label{teo_pesi_uguali}
Let $\mathcal D \subset \Omega_m^k$ be a set of $n$ nodes and let $f$ its indicator function, as in Eq. \eqref{ind}. Let $S \subset S_m$ be a monomial set such that the pair $(\mathcal D, \operatorname{Span}(S))$ is correct 
and let $w_S$ as in Proposition \ref{pr:int-cub-r}.

Let $A$ be the set:
$$ A=\left \{ z^\alpha \in\mathbb T, \  \alpha \in \mathbb Z^k_{\ge 0}  \ \left|  \ \int_{\mathbb C^k} z^\alpha \ d\lambda = \frac {m^k} n \ b_{\overline{\alpha}}\right. \right\} \ .$$

Then
$w_S= \frac 1 n 1_n$  if and only if $ S \subseteq A$.
\end{theorem}
\begin{proof}
From Proposition \ref{pr:int-cub-r} and Eq.~(\ref{ind}) (restricted to the $\alpha$ such that $z^\alpha \in S$)  it holds:
$$
 w_S= \left(X_{\mathcal D,S}^t\right)^{-1} \left[ \int z^\alpha \ d\lambda\right]_{z^\alpha\in S} \quad \textrm{and} \quad \left[ b_{\overline \alpha} \right]_{z^\alpha\in S}=\frac 1{m^k} X_{\mathcal D,S}^t\ 1_n
$$
 If $ w_S =\frac 1 n 1_n$  then
$$\left [\int z^\alpha \ d \lambda \right ]_{z^\alpha \in S}  = \frac 1 n X_{\mathcal D,S}^t\ 1_n=\frac{m^k}n \left[ b_{\overline \alpha} \right]_{z^\alpha\in S}$$
so that
$S \subseteq A$.

Vice-versa, if  $S \subseteq A$, for $z^\alpha \in S$, it holds:
$$  \left[ b_{\overline \alpha} \right]_{z^\alpha\in S} = \frac n {m^k}  \left[ \int z^\alpha \ d\lambda   \right]_{z^\alpha\in S}=
 \frac {n} {m^k} X_{\mathcal D,S}^t w_S
$$
and, from $\left[ b_{\overline \alpha} \right]_{z^\alpha\in S}=\frac 1{m^k} X_{\mathcal D,S}^t\ 1_n$, it follows $w_S=\frac 1 n 1_n$, being $X_{\mathcal D,S}$ invertible.

\end{proof}

Given a set $\mathcal D \subseteq \Omega_m^k$ with $n$ points,   Theorem~\ref{teo_pesi_uguali} suggests an algorithm for finding, if there exists, a cubature rule with nodes $\mathcal D$ and equal weights.
Fix $X^{(0)}=1_n$ and $S=\{1\}$. At the $r$-th step   an element $z^\alpha$ of $A \cap  \mathbb Z_m^k$ is considered. If the vector $v=[z^\alpha(d)]_{d \in \mathcal D}$ is such that the matrix $\left[X^{(r-1)},v \right]$ has full rank, then $v$ is added to the matrix $X^{(r-1)}$ for obtaining the new matrix $X^{(r)}=\left[X^{(r-1)},v\right]$ and $z^\alpha$ is added to $S$. Otherwise a different element of $A \cap  \mathbb Z_m^k$ is considered. 

The algorithm stops when  a square non singular matrix $X^{(n-1)}$ is computed or  if all the elements of $A \cap  \mathbb Z_m^k$ are analysed. In the former case the basis $S$ is such that the cubature rule $(\mathcal D, S)$ has equal weights. In the latter case there not exists any basis such that the associated cubature rule with node  $\mathcal D$ has  equal weights.
Notice that the algorithm stops because the elements of $A \cap \mathbb Z_m^k$ are finite.

\vspace{2 mm}

The following theorem characterises the precision basis of cubature rules with equal weights. 

\begin{theorem}\label{prec_pesi_uguali}
Let $\mathcal D \subset \Omega_m^k$ be a set of $n$ nodes and let $f$ its indicator function, as in Eq. \eqref{ind}. 
Let $S \subset S_m$ be a monomial set such that the pair $(\mathcal D, \operatorname{Span}(S))$ is correct and let $w_S$ as in Proposition \ref{pr:int-cub-r}.

If $w_S=\displaystyle  \frac 1 n 1_n$  then the set $A$ defined in Theorem~\ref{teo_pesi_uguali}
 is the precision basis $\mathcal B$ for $(\mathcal D,w_S)$.

\end{theorem}
\begin{proof}
From  Eq.~(\ref{ind}), for each $\alpha \in[0,\dots, m-1\
]^k$, $b_{\overline \alpha} =  \displaystyle \frac 1 {m^k} \sum_{d \in {\mathcal D}} z^{\alpha}(d)$, that is $\sum_{d \in {\mathcal D}} z^{\alpha}(d)= m^k b_{\overline \alpha} $.

For each $\alpha \in \mathbb Z^k_{\ge 0}$ it holds $z^\alpha(d)=z^{[\alpha]_m}(d)$. Then:
$$
\sum_{d\in \mathcal D} z^\alpha(d)=\sum_{d\in \mathcal D} z^{[\alpha]_m}(d)= m^k \ b_{\overline \alpha}
$$
Let $w = \frac 1 n 1_n$ be the weights. Then the precision space for $(\mathcal D, \frac 1 n 1_n)$ is the largest set of monomials for which
\begin{equation*}
\int z^\alpha \ d\lambda = \frac 1 n\sum_{d \in \mathcal D} \ z^\alpha(d) \quad \textrm{i.e.} \quad \int z^\alpha \ d\lambda = \frac{m^k} n b_{\overline\alpha}
\end{equation*}
that is $A$.
\end{proof}

The following result {describes the behaviour of a} the cubature rule with equal weights, when it is applied to {the integral of} monomials of the form $z^\alpha\overline z^\beta$. 
The connection to evaluation of moments of a distribution is evident. 

\begin{theorem} \label{monomio_coniugato}
Let  $(\mathcal D, w_s)$ be a cubature rule with $w_s =\frac 1 n 1_n$. The cubature rule $(\mathcal D, w_s)$ is exact for  $z^\alpha$
\begin{enumerate}
\item   if and only if $(\mathcal D, w_s)$ is exact for  $   \overline z^\alpha$.
\item  if and only if $(\mathcal D, w_s)$ is exact for  $z^{\alpha +\gamma}\, \overline z^{\gamma} $ for 
each $\gamma \in \mathbb Z^k$ such that  $\int_{\mathbb C^k} z^{\alpha +\gamma} \, \overline z^{\gamma} \, d\lambda = \int_{\mathbb C^k} z^\alpha \, d \lambda$.
\end{enumerate}
\end{theorem}
\begin{proof}
\begin{enumerate}
\item If $(\mathcal D, w_s)$ is exact for $z^\alpha$, then 
$$\int_{\mathbb C^k} z^\alpha d\lambda = \frac 1 n \sum_{d \in \mathcal D} z^\alpha (d) \ ,$$
and so
\begin{eqnarray*}
\int_{\mathbb C^k} \overline z^\alpha \ d\lambda
= \int_{\mathbb C^k}  \overline {z^\alpha} \ d\lambda
= \overline{\int_{\mathbb C^k}  z^\alpha \ d\lambda}
= \frac 1 n \overline{\sum_{d \in \mathcal D} z^\alpha (d) } 
=\frac 1 n  \sum_{d \in \mathcal D} \overline {z ^\alpha(d)} = \frac 1 n  \sum_{d \in \mathcal D} \overline z ^\alpha(d)
\end{eqnarray*}
and we conclude that $(\mathcal D, w_s)$ is exact for $\overline z^\alpha$. The vice-versa is analogous.

\item Let $(\mathcal D, w_s)$ be exact for $z^\alpha$. By assumptions $\int_{\mathbb C^k}  z^{\alpha +\gamma} \overline z^{\gamma} \ d\lambda = \int_{\mathbb C^k}  z^\alpha \ d \lambda$  and so 
\begin{align*}
\int_{\mathbb C^k} z^{\alpha +\gamma} \overline z^{\gamma} \ d\lambda 
& = \int_{\mathbb C^k}  z^\alpha \ d \lambda 
= \frac 1 n  \sum_{d \in \mathcal D}  z ^\alpha(d) = \frac 1 n  \sum_{d \in \mathcal D} |z^\gamma(d)|^2 z ^\alpha(d) \\
& =  \frac 1 n  \sum_{d \in \mathcal D} z^{\alpha + \gamma}(d) \, \overline z ^\gamma(d)  \ ,
\end{align*}
where the third equality is due to the fact that $ |z^\gamma(d)|^2 =1$. We conclude that $(\mathcal D, w_s)$ is exact for $ z^{\alpha +\gamma} \overline z^{\gamma}$. 
Furthermore, from item 1 it follows that, since the cubature rule is exact for $ z^{\alpha +\gamma} \overline z^{\gamma}$, it is also exact for $ \overline{z^{\alpha +\gamma} \overline z^{\gamma}} $, that is for 
$ \overline z^{\alpha +\gamma}z^{\gamma}$. 

The vice-versa is analogous.
\end{enumerate}
\end{proof}

\subsection{Gaussian distribution}\label{Sc:Gauss}
In this section we characterise the cubature rules  $(\mathcal D, S)$ with equal weights and  the Gaussian distribution. First of all, we present some results about integration with respect to the  Gaussian measure.

\subsubsection{Gaussian measure}

Let $Z^t=(Z_1, \dots, Z_p) $ be a  $p$-variate Gaussian complex random variable.
Let $Z_k=X_k + \mathbf{i} Y_k$, $k=1,\dots,p$, then the vector of real and imaginary parts $(X_1,Y_1,\dots, X_p,Y_p)$ is  a    $2p$-variate Gaussian real random vector. We assume the following relations among the expected values of the real and imaginary parts of the $Z$ variables.
\begin{eqnarray*}\label{conditions}
E(X_k)=E(Y_k)=0 && \\
E(X_j X_j)=E(Y_jY_j)=\frac{\sigma^2_j}{2} && E(X_j Y_j)=0 \\
E(X_j X_k)=E(Y_jY_k)=\frac{\alpha_{jk}}{2} \textrm{ for }  j\ne k && E(X_j Y_k)=-E(X_kY_j)=-\frac{\beta_{jk}}{2} \textrm{ for }  j > k
\end{eqnarray*}
for $j,k=1,\dots,p$.

We denote by  $\Sigma$ the matrix $E(Z \overline Z^t) = [ E(Z_j\overline Z_k)]_{j,k=1,\dots,p}$.
Then, from the previous conditions,
 \begin{eqnarray*}
\Sigma_{jk}=E(Z_j\overline Z_k) =   \begin{cases}
\sigma_k^2 & \text{ if }  j=k\\
\alpha_{jk}+\mathbf{i} \beta_{jk} & \text{ if } j< k\\
\alpha_{jk}-\mathbf{i} \beta_{jk} & \text{ if } j> k  \ .
\end{cases} 
\end{eqnarray*}
The probability density function of the zero mean $p$-variate complex Gaussian distribution is given by (see e.g. \cite{Goodman1963})
\begin{eqnarray}\label{Gauss_distr}
p(z)=\frac{1}{\pi^p \det (\Sigma)} \exp(-\overline z^t \Sigma^{-1} z)
\end{eqnarray}
where $z=[z_1,\dots,z_p]^t$.

We consider the complex measure $\nu$ such that $d\nu= p(z) \ d\mu$, where $\mu$ is the $\sigma$-finite measure of $\mathbb C^p$ identifiable with the Lebesgue measure of $\mathbb R^{2p}$, and so from the results presented in Appendix~\ref{Sc:Complex_Integ},  since $z_k=x_k+\mathbf{ i } y_k$, $k=1\dots p$, we have
$\int_{\mathbb C^p} f(z) \ d\nu = \int_{\mathbb C^p} f(z) p(z) \ d \mu =\int_{\mathbb R^{2p}} f(x,y) p(x,y) \ dx \, dy $, where 
$x=[x_1,\dots,x_p]^t$ and $y=[y_1,\dots,y_p]^t$.

We denote by $\nu(n_1,m_1,\dots,n_p,m_p)$  the moment:
\begin{equation*}
  \nu(n_1,m_1,\dots,n_p,m_p)=\frac{1}{\pi^p \det (\Sigma)} \int_{\mathbb C^p} z_1^{n_1}\  \overline z_1^{m_1}\ z_2^{n_2}\  \overline z_2^{m_2} \cdots z_p^{n_p} \ \overline z_p^{m_p}  \exp(-\overline z^t \Sigma^{-1} z) \ dz_1 \cdots dz_p
\end{equation*}
In the vector $(n_1,m_1,\dots,n_p,m_p)$, consisting of the exponents of the moments,  the $n_j$ indices are in odd entries and refer to $Z_j$ while the $m_j$ indices are in even entries and refer to $\overline Z_j$.

In~\cite{FassinoPistoneRiccomagnoRogantin2017} the following theorem on the null moments is shown.

\begin{theorem}[Null moments]\label{th:conditions}
The conditions for nullity of the moment $\nu(n_1,m_1,\dots,n_p,m_p)$ depend on the structure of independence of the variables.
\begin{enumerate}
\item If no variable $Z_j$, $j=1,\dots,p$, is independent from all the others, then the moment  $\nu(n_1,m_1,\dots,n_p,m_p)$ is zero if 
$$\sum_{j=1}^p n_j\ne \sum_{j=1}^p m_j\ ;$$
\item If there exists a subset of variables $Z_r$, $r\in R \subset \{1,\dots,p\}$ such that $Z_r$ is independent from all the others, then  $\nu(n_1,m_1,\dots,n_p,m_p)$ is zero if 
there exists $r \in R $ such that
$$ n_r\ne m_r \qquad \textrm{or} \qquad \sum_{j=1,j\notin R}^p n_j\ne \sum_{j=1,j\notin R}^p m_j \ .$$
\item If there exist $q$ subsets of variables $R_1\dots, R_q$ with variables dependent within each subset and independent between subsets, then  $\nu(n_1,m_1,\dots,n_p,m_p)$ is zero if there exists 
$h \in \{1,\dots,q\} $ such that 
$$ \sum_{j\in R_h}^p n_j\ne \sum_{j\in R_h} m_j \ .$$
\end{enumerate}
\end{theorem}
From Theorem~\ref{th:conditions} we obtain the value of the integral of the monomials in $\mathbb T$ with respect to the Gaussian distribution.
\begin{corollary}\label{cor:integral}
Let $z^\alpha=z_1^{n_1}\cdots z_p^{n_p}$ a monomial in $\mathbb T$. Then 
\begin{eqnarray*}
\begin{cases}
\int_{\mathbb C^p} z^\alpha p(z) \ dz =0 & \text{if there exists }  r \in \{1,\dots,p\} \ s.t. \ n_r \neq 0\\ \\
\int_{\mathbb C^p} z^\alpha p(z) \ dz =1 &  \text{if }  n_1=\dots = n_p= 0
\end{cases}
\end{eqnarray*}

\end{corollary}
\begin{proof}
In this case $m_1=\dots=m_p=0$, and so, if there exists an exponent $n_r\neq 0$,   from Theorem~\ref{th:conditions} we have $ \nu(n_1,0,\dots,n_p,0)=\displaystyle \frac{1}{\pi^p \det (\Sigma)} \int_{\mathbb C^p} z_1^{n_1}\   \cdots z_p^{n_p} \ p(z) \ dz_1 \cdots dz_p =0$ and the first part of the thesis follows. The second part is an obvious result. 
\end{proof}

%
%
%
%
%

\subsubsection{Cubature rules  with equal weights for Gaussian distribution}
Let $\nu$ be the Gaussian distribution.  Since, in this case $\int_{\mathbb C^k} z^\alpha \ d\nu =0$, if $\alpha \neq (0,\dots,0)$,  Theorem~\ref{teo_pesi_uguali} can be reformulated as follows.

\begin{theorem}\label{normale_uguali}
Let $\nu$ be the Gaussian distribution. Let $\mathcal D \subset \Omega_m^k$ be a set of $n$ nodes and let $f$ its indicator function, as in Eq. \eqref{ind}. Let $S \subset S_m$ be a monomial set such that the pair $(\mathcal D, \operatorname{Span}(S))$ is correct and let $w_S$ as in Proposition \ref{pr:int-cub-r}.
Then   $w_S= \frac 1 n 1_n$  if and only if $ S \cap \operatorname{Supp}(f) = \{ 1\}$.
\end{theorem}
\begin{proof}
From Theorem~\ref{teo_pesi_uguali} we have that $w_S= \frac 1 n 1_n$  if and only if $S\subset A$, where
$$ A=\left \{ z^\alpha \in\mathbb T, \  \alpha \in \mathbb Z^k_{\ge 0}  \ \left|  \ \int_{\mathbb C^k} z^\alpha \ d\nu = \frac {m^k} n \ b_{\overline{\alpha}}\right. \right\} \ ,$$
and so we first describe $A$ for the Gaussian distribution.

When $\alpha=[0,\dots,0]$ we have that $\int z^{\alpha} d\nu =1 $ and $b_{\overline \alpha}=\frac n {m^k}$ and so $1$ belongs to $A$. Furthermore, if $\alpha \neq [0,\dots, 0]$, then $\int_{\mathbb C^k} z^\alpha d\nu =0$, and so the set $A$ is given by
\begin{eqnarray*} 
A=\{1\} \cup \left \{ z ^\alpha \in \mathbb T ,\ \alpha \in \mathbb Z^k_{> 0}\ \left|  \   b_{\overline\alpha} =0\right. \right\}=\{1\} \cup \left \{ z ^\alpha \in \mathbb T ,\ \alpha \in \mathbb Z^k_{> 0}\ \left|  \   b_{[\alpha]_m} =0\right. \right\}
\end{eqnarray*}
since $b_{[\alpha]_m}=\overline{ b_{\overline\alpha} }$. 
From Equation~(\ref{ind}), we have that $b_{[\alpha]_m}=0$ if and only if   $z^{[\alpha]_m} \notin \operatorname{Supp}(f)$, and thus 
\begin{equation*}
A =\{1\} \cup   \left \{z^\alpha \in \mathbb T, \alpha \in \mathbb Z^k_{> 0}\;|\;  z^{[\alpha]_m} \notin \operatorname{Supp}(f) \right \}  \ .
\end{equation*}
We conclude that  $S\subset A$  if and only if $ S \cap \operatorname{Supp}(f) = \{ 1\}$.
\end{proof}

The following theorem characterises  the precision basis of a cubature rule with equal weights, with respect to  the gaussian distribution.

\begin{corollary}\label{pesi_uguali}
Let $\nu$ be the Gaussian distribution. Let $\mathcal D \subset \Omega_m^k$ be a set of $n$ nodes and let $f$ its indicator function, as in Eq. \eqref{ind}. Let $S \subset S_m$ be a monomial set such that the pair $(\mathcal D, \operatorname{Span}(S))$ is correct and let $w_S$ as in Proposition \ref{pr:int-cub-r}.
If $w_S= \frac 1 n 1_n$  then  
$\left \{1\right\} \cup \left \{z^\alpha \in \mathbb T\,|\, z^{[\alpha]_m} \notin \operatorname{Supp}(f) \right \} $ is the precision basis $\mathcal B$ for $(\mathcal D,w_S)$.
\end{corollary}

\smallskip\noindent
Given a set of nodes $\mathcal D$, it is possible to check if there exists a basis $S$ so that the corresponding cubature rule $(\mathcal D, S)$ has equal weights in an easier way than in the general case, since it is sufficient to   consider only  monomials  $t \in \mathbb Z_m^k \setminus \operatorname{Supp}(f)$, as Corollary~\ref{pesi_uguali} shows that $ S \cap \operatorname{Supp}(f) =\{1\}$ in order to have equal weights.

\begin{proposition}
Let $\nu$ be the Gaussian distribution.   If $\mathcal D$ is a regular fraction,  for each  basis $S$ such that  the pair $(\mathcal D, \operatorname{Span}(S))$ is correct, the corresponding cubature rule has equal weights.
\end{proposition}
\begin{proof}
 If $\mathcal D$ is a regular fraction, for each monomial  $s\in \mathbb T$, the vector $[s(d)]_{d \in \mathcal D}$ is equal to $\gamma  1_n$, for a given $\gamma\in \mathbb C$, or  orthogonal  to $1_n$. Since $X_{\mathcal D, S}$ is a non singular matrix whose first column is the vector $1_n$, for each monomial basis $S=\{z^\alpha\}$, the columns of  $X_{\mathcal D, S}$, except the first one,  are orthogonal to $1_n$. It follows that $\Sigma_{d\in \mathcal D} z^\alpha(d) =0$ for each  $z^\alpha \in S\setminus \{1\}$ and so  $z^\alpha \notin \operatorname{Supp}(f)$.
We conclude that, if $\mathcal D$ is a regular fraction,  $S \cap \operatorname{Supp} (f)= \{1\}$ for each possible basis $S$ and so the corresponding cubature rule has equal weights.
\end{proof}

\medskip
The following example shows that a cubature rule can have equal weights even if $\mathcal D$ is not a regular fraction of $\Omega_m^k$.
\begin{example}{\rm
Let $\mathbb C[z_1,\dots, z_4]$ be the polynomial ring with indeterminates $z_1,\dots,z_4$ and let $\mathcal D$ be the set of nodes contained in $\Omega_2^4$,
\begin{align*}
\mathcal D =& \,  \{(1,1,1,1),(1,1,-1,1),(1,-1,1,-1),(1,-1,-1,-1), \\
& \quad (-1, 1, 1, 1), (-1, -1, -1, 1), (-1, -1, -1, -1), (-1, 1, 1, -1)\}
\end{align*}
whose indicator function is $f = (2 + z_2z_3 + z_2z_4-z_1z_2z_3 + z_1z_2z_4)/4$.
The set $\mathcal D$ is not a regular fraction since $[z_1z_2z_4(d) ]_{d\in \mathcal D} = [1, 1, 1,1, -1, 1, -1,1]^t $,  that is $[z_1z_2z_4(d) ]_{d\in \mathcal D} $ is not orthogonal nor parallel to $1_8$.

Given the monomial set $S = [1, z_1, z_2, z_3, z_4, z_1z_2, z_3z_4, z_1z_3z_4]$, the weights $w_S$ are the solution of the linear system $X_{ \mathcal D, S}^t w_S =  [\int sd\nu]_{s\in S}$, where $\int s \,d\nu =0 $ if $s\neq 1$ and $\int s \,d\nu=1$ if $s=1$. Since the columns of the matrix $X_{ \mathcal D, S}$, except the first one, are orthogonal to $1_8$, the vector $w_S= \frac 1 8 1_8$ is the solution of the previous linear system, even if the set of nodes is not a regular fraction.
}\end{example}

\medskip
The following example illustrates a cubature rule with equal weights and its precision basis.

\begin{example}{\rm
Let $k=2$, $m=4$ and let $\omega_0=1$, $\omega_1= \mathbf{i}$, $\omega_2=-1$ and $\omega_3=-\mathbf{i}$ be the fourth root of the unit. In this case, the full factorial is $\mathcal F=[\omega_0, \omega_1, \omega_2, \omega_3 ]^2$.
Let  $\mathcal D= \{ (1,1), (\mathbf{i},-\mathbf{i}), (-1,\mathbf{i}), (-\mathbf{i},-1)   \} $ be the set of nodes with indicator function 
\begin{eqnarray*}
 f= \frac1 8 \left(2+z_1z_2+(1+\mathbf{i})z_1z_2^2+ (1-\mathbf{i})z_1^2z_2 - \mathbf{i}z_1z_2^3 + \mathbf{i} z_1^3z_2 +(1+\mathbf{i})z_1^2z_2^3+(1-\mathbf{i})z_1^3z_2^2+z_1^3z_2^3\right) \, . 
 \end{eqnarray*}
Given the monomial set $S = [1, z_2, z_1, z_2^3]$, we have $S \cap  \operatorname{Supp}(f)=\{1\}
$ 
and so the nodes of the cubature rule $(\mathcal D, S)$ are  $w_d = \frac 1 4$, for all $ d \in \mathcal D$.
From Corollary~\ref{pesi_uguali} it follows that the precision basis is 
$$\mathcal B_{\mathcal D,S} =  \{1\} \cup \left \{ z^\alpha \in \mathbb T \;|\; [\alpha]_4 \in \{ (0,1), (1,0), (0,2), (2,0),(0,3), (3,0), (2,2) \} \right \} \, . $$
From Theorem~\ref{monomio_coniugato} it follows that the cubature rule is also exact for
$\left \{{\overline z}^\alpha  \;|\; z^\alpha \in\mathcal B_{\mathcal D,S} \right \}$. Furthermore, since for each $\alpha \neq (0,\dots,0)$ and for each $\gamma \in \mathbb Z_{\ge 0}^4$ we have
$\int_{\mathbb C^k} z^\alpha d \nu = 0 =\int_{\mathbb C^k} z^{\alpha +\gamma} \, \overline z^{\gamma} \, d\nu$ , the cubature rule is also exact for $\left \{\overline z^{\alpha+\gamma} z^\gamma, z^{\alpha+\gamma}\overline z^\gamma \;|\; z^\alpha \in  \mathcal B_{\mathcal D,S}  , \ \gamma \in \mathbb Z_{\ge 0}^4\right \}$.
 }\end{example}

The following example shows the case of a set of nodes $\mathcal D$ which does not generate any cubature rule  with equal weights.

\begin{example}{\rm
Let $k=2$, $m=3$ and let $\omega_0=1$, $\omega_1= \cos(2\pi/3) +\mathbf{i} \sin(2\pi/3)$ and  $\omega_2=\overline{\omega}_1$  be the third root of the unit.

Let $\mathcal D =\{(1,\omega_2),\;(\omega_2,\omega_1) \}$ be the set of nodes with indicator function
$$ f= \frac1 9 \left(2 -z_2 - \omega_2z_1 - z_2^2 - \omega_2z_1z_2 - \omega_1z_1^2 + 2\omega_2z_1z_2^2 + 2\omega_1z_1^2z_2- \omega_1z_1^2z_2^2\right) \,.$$
Since $\operatorname{Supp}(f)=\{z^\alpha \,|\, \alpha \in  \mathbb Z_3^2 \}$, there not exist a monomial basis  $S\subset \mathbb Z_3^2$ such that $S\cap \operatorname{Supp}(f) = \{1\}$ and so there not exist a cubature  rule  $(\mathcal D,S)$ with equal weights.
}
 \end{example}

\section*{Appendix~A. Complex integration}\label{Sc:Complex_Integ}
Let $\mathcal M$ be a $\sigma$-algebra in a set $X$ and let $\{E_k\}$ be a countable partition of $E$, that is $E=\cup_k E_k$ and $E_k \cap E_j =\emptyset$, if $k\neq j$. A complex measure $\lambda$ on $\mathcal M$ is a complex-valued function on $\mathcal M$ such that
$$ \lambda (E) = \sum_{k=1}^\infty \lambda(E_k)  < +\infty.$$
The total variation $|\lambda|$ of $\lambda$ is a real positive measure defined as
$$ |\lambda| (E)= \sup_{\{E_k\}_{k=1}^\infty}\sum_{k=1}^\infty |\lambda(E_k)| \qquad \text{ for all }  E \in \mathcal M $$
where ${\{E_k\}_{k=1}^\infty}$ is a generic partition of $E$. 

The following theorem is the Lebesbue-Radon-Nikodym Theorem presented in~\cite[Th. 6.10]{Rudin1987}.
\begin{theorem}\label{Lebesbue-Radon-Nikodym}
Let $\mu$ be a positive $\sigma$-finite measure on a $\sigma$ algebra $\mathcal M$ in a set $X$, and let $\lambda$ be a complex measure on $\mathcal M$.
\begin{description}
\item{(a)} There is then a unique pair of complex measures $\lambda_a$ and $\lambda_s$ on $\mathcal M$ such that 
$$ \lambda= \lambda_a +\lambda_s \qquad \lambda_a {\ll} \mu \qquad \lambda_s {\perp }\mu \, .$$
\item{(b)} There is a unique $ h \in L^1(\mu)$, called the Radon-Nikodym derivative w.r.t. $\mu$, such that
$$ \lambda_a(E) = \int_E h \ d\mu \, . $$
\end{description}
\end{theorem}
The choice  $\mu=|\lambda|$ gives the following theorem~\cite[th. 6.12]{Rudin1987}.
\begin{theorem}\label{Th_dlambda}
Let $\lambda$ be a complex measure on a $\sigma$-algebra $\mathcal M$ in $X$. Then there is a function  $h\in L^1(|\lambda|)$, called the Radon-Nikodym derivative w.r.t. $|\lambda|$, such that $|h(x)| =1$ for all $x \in X$ and such that 
$ \lambda(E) = \int_E h \ d|\lambda|$ or, equivalently, that $d\lambda= h\ d|\lambda|$. 
\end{theorem}
From Theorem~\ref{Th_dlambda} it follows that $\lambda(X)  = \int_X h d|\lambda|$; furthermore, it is possible to define 
$$ \int_X f \ d\lambda \stackrel{def}{=} \int_X f h \ d|\lambda| \, .$$
Later on, we consider a special case. Let $\mu$ be a positive real measure on $\mathcal M$ and let $g: X \rightarrow \mathbb C$ be a function in $L^1(\mu)$. We can define a complex measure $\lambda$ on  $\mathcal M$ in the set $X$ as follows:
$$\lambda(E)=\int_E g \ d\mu \quad \text{ for all } \, E \in \mathcal M \, .$$
Since in this case $d\lambda= g \ d\mu$, from Theorem~\ref{Th_dlambda} we have $g \ d\mu = d\lambda = h\  d|\lambda|$, with $|h|=1$, and so  $\overline h g \ d\mu = \overline h h d|\lambda| = d|\lambda|$. We conclude that, in this case,
 \begin{eqnarray}\label{integrale}
\int_X f \ d\lambda =\int_X f h \ d|\lambda|  = \int_X f h \overline h g \ d\mu = \int_X f g |h|^2 \ d\mu = \int_X f g \ d\mu \, .
 \end{eqnarray}


\bibliography{FassinoRiccomagnoRogantin._JAS_paper}
\bibliographystyle{plain}
\end{document}